\documentclass[reqno,12pt]{amsart}
\usepackage{amssymb,amsmath,amsthm}

\usepackage{enumerate}
\def\Z{\mathbb{Z}}
\def\Q{\mathbb{Q}}

\def\H{\mathbb{H}}

\def\GL{{\rm GL}}

\newcommand{\pfrac}[2]{\left(\frac{#1}{#2}\right)}
\newcommand{\pMatrix}[4]{\left(\begin{matrix}#1 & #2 \\ #3 & #4\end{matrix}\right)}
\renewcommand{\pmatrix}[4]{\left(\begin{smallmatrix}#1 & #2 \\ #3 & #4\end{smallmatrix}\right)}
\renewcommand{\bar}[1]{\overline{#1}}
\newcommand{\floor}[1]{\left\lfloor #1 \right\rfloor}
\newcommand{\tfloor}[1]{\lfloor #1 \rfloor}

\newtheorem{theorem}{Theorem}

\newtheorem{proposition}[theorem]{Proposition}

\theoremstyle{remark}
\newtheorem*{remark}{Remark}

\newtheorem{example}{Example}

\numberwithin{equation}{section}

\begin{document}


\title[Hecke grids and congruences]{Hecke grids and congruences for\\weakly holomorphic modular forms}

\date{\today}

\author{Scott Ahlgren}
\address{Department of Mathematics\\
University of Illinois\\
Urbana, IL 61801} 
\email{sahlgren@illinois.edu} 

\author{Nickolas Andersen}
\address{Department of Mathematics\\
University of Illinois\\
Urbana, IL 61801} 
\email{nandrsn4@illinois.edu}
 


\begin{abstract}  
Let $U(p)$ denote the Atkin operator of prime index $p$.
 Honda and Kaneko proved infinite families of congruences of the form $f\big|U(p) \equiv 0 \pmod{p}$ for weakly holomorphic modular forms of low weight and level and primes $p$ in certain residue classes, and  conjectured the existence of similar 
 congruences modulo higher powers of $p$.  Partial results on  some of these conjectures were proved recently by Guerzhoy.
We construct infinite families of weakly holomorphic modular forms on the Fricke groups $\Gamma^*(N)$ for $N=1,2,3,4$ and describe explicitly the action of the Hecke algebra on these forms. As a corollary,  we obtain strengthened versions of all of the congruences conjectured by Honda and Kaneko.
\end{abstract}

\thanks{The first author was supported by a grant from the Simons Foundation (\#208525 to Scott Ahlgren).}


\maketitle


\section{Introduction}
For a prime number $p$, let $U(p)$ denote Atkin's operator, which acts on power series via 
\[
	\left(\sum a(n)q^n\right)\big |U(p):=\sum a(pn )q^n.
\]
In recent work, Honda and Kaneko \cite{HK} generalize a theorem of Garthwaite \cite{Garthwaite} in order to establish infinite families of congruences of the form 
\[
	f\big|U(p)\equiv 0\pmod p
\]
for weakly holomorphic modular forms of low weight and level. 
For example,  it is shown that for any prime $p\equiv 1\pmod 3$ and any $k\in\{4, 6, 8, 10, 14\}$ we have
\begin{equation} \label{hk1}
	\frac{E_k(6z)}{\eta^4(6z)}\big |U(p)\equiv 0\pmod p.
\end{equation}
For another example, if $p\equiv 1\pmod 4$,   $k\in\{4, 6\}$, and   $f\in M_k(\Gamma_0(2))$ has $p$-integral Fourier expansion, then it is shown that
\begin{equation} \label{hk2}
	\frac{f(4z)}{\eta^2(4z)\eta^2(8z)}\big |U(p)\equiv 0\pmod p.
\end{equation}
Honda and Kaneko conjecture that these extend to congruences modulo higher powers of $p$. For example,
they conjecture that for any $p\equiv 1\pmod 3$, the congruence \eqref{hk1} can be replaced by
\begin{equation} \label{hk1powers}
	\frac{E_k(6z)}{\eta^4(6z)}\big |U(p^n) \equiv 0\pmod {p^{n(k-3)}}\quad \text{for any $n\geq 1$}.
\end{equation}

In recent work, Guerzhoy \cite{Guerzhoy} studies the conjectures \eqref{hk1powers} using the $p$-adic theory of weak harmonic Maass forms.
In the case when $k=4$, he shows that if $p\equiv 1\pmod 6$, then there exists an integer $A_p$ such that for all $n$ we have
\begin{equation} \label{g1}
	\frac{E_4(6z)}{\eta^4(6z)}\big |U(p^n) \equiv 0\pmod {p^{n-A_p}},
\end{equation}
and that if $p\equiv 5\pmod 6$, then there exists an integer $A_p$ such that for all $n$ we have
\begin{equation} \label{g2}
	\frac{E_4(6z)}{\eta^4(6z)}\big |U(p^n) \equiv 0\pmod {p^{\lfloor\frac n2\rfloor -A_p}}.
\end{equation}

In this paper, we show that the congruences conjectured by Honda and Kaneko  
result from the existence of ``Hecke grids" of 
weakly holomorphic modular forms on Fricke groups.  These are infinite families of forms on which the Hecke algebra acts in a systematic way. These are similar to the well-known grid of Zagier \cite{Zagier:2002} which encodes the traces of singular moduli; a similar Hecke action on this grid \cite{Ahlgren:2012} explains the many congruences among these traces.

Since the congruences are straightforward consequences of identities involving the Hecke operators we will focus here on the identities themselves.
As an example of the results, we consider the case related to \eqref{g1} and \eqref{g2}.  Using Theorem~\ref{thm:level1} below with $k=r=4$, we see that there is an infinite family of forms $F_d(z)\in M_2^!(\Gamma_0(36))$ with $p$-integral coefficients, and with $F_1(z)=\frac{E_4(6z)}{\eta^4(6z)}=\sum a_1(n)q^n$, such that 
\begin{equation}\label{conj1}
	F_1 \big| T(p^{n}) = 
	\begin{cases} 
		p^{n} F_{p^ n } & \text{ if } p^n \equiv 1 \pmod6,\\
		p^{n} F_{p^{n}} + a_{1}(p^n) \eta^4(6z) & \text{ if } p^{n} \equiv 5 \pmod6.
	\end{cases}
\end{equation}
Using relations among the Hecke operators (we sketch the proof in Section~\ref{sec:level-1} below), we conclude that 
\begin{equation}\label{conj2}
	F_1 \big|U(p^ n ) \equiv  
	\begin{cases} 
		0 \pmod {p^n}  & \text{ if } p \equiv 1 \pmod6,\\
	 	0 \pmod {p^{\lfloor \frac n2\rfloor}} & \text{ if } p \equiv 5 \pmod6.
	\end{cases}
\end{equation}
In other words, \eqref{g1} and \eqref{g2} are true with $A_p=0$ for every $n$.

In some cases, \eqref{conj1} and \eqref{conj2} can be strengthened.  For example, if  $G_1(z)=\frac{E_6(6z)}{\eta^4(6z)}$, Theorem~\ref{thm:level1} gives a family $G_d$ with the property that $G_1\big |T(p^n)=p^{3n} G_{p^n}$ for all $p\geq 5$.  We conclude that $G_1\big |U(p^n)\equiv 0\pmod{p^{3n}}$, as shown in \cite{Guerzhoy}. This phenomenon will occur whenever the parameter $\ell$
in Theorem~\ref{thm:level1} is non-zero.

In a similar way, we obtain strengthened versions of the other conjectures in \cite{HK}.
For example, consider the congruence  \eqref{hk2} in the case $k=4$.
Any form $f\in M_4(\Gamma_0(2))$ can be written uniquely as the sum $f=a f^++b f^-$, 
where $f^+(z)=1+48q+\dots$ and $f^{-}(z)=1-80q+\dots$ are eigenforms for the Fricke involution $f(z)\mapsto 2^{-2}z^{-4}f(-1/2z)$.

Define
\[
	F_1^+(z):=\frac{f^-(4z)}{\eta^2(4z)\eta^2(8z)}, \qquad F_1^-(z):=\frac{f^+(4z)}{\eta^2(4z)\eta^2(8z)}=\sum a_1^-(n)q^n.
\]
Using Theorem~\ref{level2thm} below, we conclude that for positive odd $d$ there are $p$-integral forms 
$F_d^{\pm}\in M_2^!(\Gamma_0(32))$ with the following properties:
For all prime powers $p^n$ we have
\begin{equation*}
	F_1^+\big|T(p^n)=p^n F_{p^n}^+.
\end{equation*}
If $p^n\equiv 1\pmod 4$ then 
\begin{equation*}
	F_1^-\big|T(p^n)=p^n F_{p^n}^-.
\end{equation*}
If $p^n\equiv 3\pmod 4$ then
\begin{equation*}
	F_1^-\big|T(p^n)=p^n F_{p^n}^-+ a_1^-(p^n)\cdot \eta^2(4z)\eta^2(8z).
\end{equation*}

We conclude as above that  
\begin{equation}\label{level2pm}
	F_1^\pm \big|U(p^ n ) \equiv  \begin{cases} 0 \pmod {p^n}  \quad&\text{if $p \equiv 1 \pmod4$},\\
	 0 \pmod {p^{\lfloor \frac n2\rfloor}}\quad&\text{if $p \equiv 3 \pmod4$.}
\end{cases}
\end{equation}
For all odd primes $p$, any $f\in M_4(\Gamma_0(2))$ having $p$-integral coefficients is a $p$-integral linear combination of $f^+$ and $f^-$. It follows that \eqref{level2pm} holds for $\frac{f(4z)}{\eta^2(4z)\eta^2(8z)}$; this establishes a stronger version of the conjecture of \cite{HK}. 

The following strengthened versions of these conjectures  for $\Gamma_0(3)$ and $\Gamma_0(4)$ arise from the identities of Theorems \ref{level3thm} and \ref{level4thm} below. Let $p\geq 5$ be prime and let $N\in \{3,4\}$. Suppose that $f\in M_4(\Gamma_0(N))$ has $p$-integral coefficients and define
\begin{gather*}
	H_3(z) := \eta^2(3z) \eta^2(9z) = q-2 q^4-q^7+5 q^{13}+4 q^{16}-7 q^{19}+\cdots, \\
	H_4(z) := \eta^4(6z) = q-4 q^7+2 q^{13}+8 q^{19}-5 q^{25}+\cdots.
\end{gather*}
Then we have
\[
	\frac{f(3z)}{H_N(z)} \equiv 
	\begin{cases}
		0 \pmod{p^n} &\text{ if } p\equiv 1 \pmod{3}, \\
		0 \pmod{p^{\tfloor{\frac n2}}} &\text{ if } p \equiv 2 \pmod{3}.
	\end{cases}
\]
Finally,  we mention that similar results will hold if the initial forms $F_1$ are replaced by other members of the grid.


\section{Preliminaries}
We begin with some brief background and a proposition about the action of the Hecke operators on the spaces in question.
It will be most natural to work with the Fricke groups 
$\Gamma^*(N)$ for $N \in \{1,2,3,4\}$ (see \cite[Section 1.6]{Kohler} for background).
For these levels, the  groups are generated by the translation
\[
	T := \pMatrix{1}{1}{0}{1}
\]
and the Fricke involution
\[
	W_N := \pMatrix{0}{-1}{N}{0}.
\]
Let $k$ be a positive integer. If $\gamma=\pmatrix abcd \in \GL_2^+(\Q)$, define the slash operator $\big|_k$ by
\[
	f \big|_k \gamma := (\det \gamma)^{k/2} (cz+d)^{-k} f\pfrac{az+b}{cz+d}.
\]
Define $\Gamma_0(M,N) := \left\{ \pmatrix abcd \in \Gamma_0(1) : M|c \text{ and } N|b \right\}$. For primes $p$, define the Hecke operator $T_k(p)$ by 
\begin{equation}
	f \big| T_k(p) := f \big| U(p) + p^{k-1} f(pz) = p^{\frac{k}{2}-1} \left( \sum_{\lambda=0}^{p-1} f \big|_k \pMatrix{1}{\lambda}{0}{p} + f \big|_k \pMatrix{p}{0}{0}{1} \right).
\end{equation}
For $(t,p)=1$, define the conjugated operator $T_k^{(t)}(p) := A_t T_k(p) A_t^{-1}$, where 
\[
	A_t := \pMatrix t001.
\]
Then
\begin{equation}
	f \big| T_k^{(t)}(p) := p^{\frac{k}{2}-1} \left( \sum_{\lambda=0}^{p-1} f \big|_k \pMatrix{1}{t\lambda}{0}{p} + f \big|_k \pMatrix{p}{0}{0}{1} \right),
\end{equation}
and if $f=\sum a_f(n) q^{n/t}$, then
\begin{equation} \label{eq:T-q-series}
	f \big| T_k^{(t)}(p) = \sum \left( a_f(pn) + p^{k-1} a_f(n/p) \right) q^{n/t}.
\end{equation}
For prime powers $p^n$ we have $T_k^{(t)}(p^n) = A_t T_k(p^n) A_t^{-1}$ and the recurrence relation
\begin{equation} \label{eq:T-recurrence}
	T_k^{(t)}(p^{n+1}) = T_k^{(t)}(p^n)T_k^{(t)}(p) - p^{k-1}T_k^{(t)}(p^{n-1}).
\end{equation}
We suppress the subscript $k$ when it is clear from context.

We say that $\nu$ is a multiplier system for a group $\Gamma$ if $\nu$ is a character on $\Gamma$ of absolute value $1$ (see \cite[Section 1.4]{Kohler} for details). Then $M_k^!(\Gamma, \nu)$ is the space of holomorphic functions $f$ on $\H$ whose poles are supported at the cusps of $\Gamma$, and which satisfy
\begin{equation}
	f \big|_k \gamma = \nu(\gamma) f
\end{equation}
for all $\gamma \in \Gamma$.

The multiplier system $\nu_\eta$ on $\Gamma^*(1)$ for the Dedekind $\eta$ function
\[
	\eta(z) := q^{\frac{1}{24}} \prod_{n=1}^\infty (1-q^n)
\]
is given by
\begin{equation} \label{eq:eta-mult}
	\nu_\eta\left( \pmatrix abcd \right) = 
	\begin{cases}
		\pfrac{d}{c}^* \exp\left( \frac{2\pi i}{24} \left( (a+d)c-bd(c^2-1)-3c \right) \right) & \text{ if $c$ is odd,} \\
		\pfrac{c}{d}_* \exp\left( \frac{2\pi i}{24} \left( (a+d)c-bd(c^2-1)+3d-3-3cd \right) \right) & \text{ if $c$ is even}
	\end{cases}
\end{equation}
(see \cite[Chapter 4]{Knopp}). The symbols $\pfrac{d}{c}^*$ and $\pfrac{c}{d}_*$ denote extensions of the Jacobi symbol to negative integers, and take the values $\pm 1$.

The following proposition describes the effect of the conjugated Hecke operators $T_k^{(t)}(p^n)$ on these spaces.

\begin{proposition} \label{prop:Tp}
Let $N\in\{1,2,3,4\}$ and suppose that $t$ is a positive integer. Suppose that $\nu$ is a multiplier system on $\Gamma^*(N)$ which takes values among the $2t$-th roots of unity, and that $\nu$ is trivial on $\Gamma_0(Nt,t)$. Then for primes $p \nmid N$ with $p^{2} \equiv 1 \pmod{2t}$, we have
\[
	T_k^{(t)}(p^{n}) : M_k^!(\Gamma^*(N), \nu) \to M_k^!\left(\Gamma^*(N), \nu^{p^{n}}\right).
\]
\end{proposition}

\begin{proof} We proceed by induction on $n$. For $n=1$, it is enough to show that for each of the two generators $\gamma$ we have
\begin{equation*}
	f \big| T^{(t)}(p) \big|_k \gamma = \nu^p(\gamma) f \big| T^{(t)}(p).
\end{equation*}
 We begin with the translation $T$. We have
\begin{align*}
	f \big| T^{(t)}(p) \big| T
		&= p^{\frac k2-1}\left(\sum_{\lambda=0}^{p-1} f \big|_{k} \pMatrix{1}{t\lambda}{0}{p} \pMatrix 1101 + f \big|_{k} \pMatrix p001 \pMatrix 1101 \right)\\
		&= p^{\frac k2-1}\left(\sum_{\lambda=0}^{p-1} f \big|_{k} \pMatrix 1p01 \pMatrix{1}{t\lambda+1-p^2}{0}{p} + f \big|_{k} \pMatrix 1p01 \pMatrix p001 \right).
\end{align*}
Define $\lambda'$ by $\lambda' \equiv \lambda + (1-p^2)/t \pmod p$ and $0\leq \lambda' \leq p-1$. Then
\begin{align*}
	f \big| T^{(t)}(p) \big| T
		&= p^{\frac k2-1}\left(\sum_{\lambda'=0}^{p-1} \nu^p(T) \, f \big|_{k} \pMatrix{1}{t\lambda'}{0}{p} + \nu^p(T) \, f \big|_{k} \pMatrix p001 \right) \\
		&= \nu^p(T) \, f \big| T_k^{(t)}(p).
\end{align*}

Since conjugation by $W_N$ interchanges $\pmatrix p001$ and $\pmatrix 100p$, we have
\begin{align*}
f \big| T^{(t)}(p) \big|_k W_N
	&= p^{\frac k2-1}\left(\sum_{\lambda=0}^{p-1} f \big|_{k}  \pMatrix{1}{t\lambda}{0}{p} \pMatrix{0}{-1}{N}{0} + f \big|_{k} \pMatrix p001 W_N \right)\\
   	&= p^{\frac k2-1}\left(\sum_{\lambda=1}^{p-1} f \big|_{k}  \pMatrix{Nt\lambda}{-1}{Np}{0} +    f \big|_{k} W_N\pMatrix{p}{0}{0}{1} +  f \big|_{k} W_N \pMatrix{1}{0}{0}{p}   \right).
\end{align*}
Define $\lambda'$ by $Nt^2\lambda\lambda'+1\equiv 0\pmod p$ and $1\leq \lambda' \leq p-1$.
Then 
\[
\pMatrix{Nt\lambda}{-1}{Np}{0}
	=\pMatrix {\frac{1+Nt^2\lambda\lambda'}p}{t\lambda}{Nt\lambda'}{p}W_N\pMatrix{1}{t\lambda'}{0}{p}.
\]
By assumption we have
\[
	\nu \pMatrix {\frac{1+Nt^2\lambda\lambda'}p}{t\lambda}{Nt\lambda'}{p} = 1.
\]
Therefore
\begin{align*}
	f \big| T^{(t)}(p) \big|_k W_N
		&= \nu(W_N) p^{\frac k2-1}\left(\sum_{\lambda'=1}^{p-1} f \big|_{k}  \pMatrix{1}{t\lambda'}{0}{p} +    f \big|_{k} \pMatrix{p}{0}{0}{1} +  f \big|_{k}  \pMatrix{1}{0}{0}{p}   \right) \\
		&= \nu(W_N) f \big| T^{(t)}(p)= \nu^p(W_N) f \big| T^{(t)}(p),
\end{align*}
since $p$ is odd and $W_N$ is an involution.

Suppose that $n\geq1$ and recall the recurrence \eqref{eq:T-recurrence} satisfied by $T^{(t)}(p^{n+1})$. By induction, the form $f\big|T^{(t)}(p^{n})\big|T^{(t)}(p)$ has multiplier system $\nu^{p^{n+1}}$ and the form $f\big|T^{(t)}(p^{n-1})$ has multiplier system $\nu^{p^{n-1}}$. Since the values of $\nu$ are $2t$-th roots of unity and $p^{2}\equiv 1 \pmod{2t}$, these systems are the same. Therefore,
\[
	T_k^{(t)}(p^{n+1}) : M_k^!(\Gamma^*(N), \nu) \to M_k^!\left(\Gamma^*(N), \nu^{p^{n+1}}\right). \qedhere
\]
\end{proof}


\section{Hecke grids on $\Gamma^*(1)$} \label{sec:level-1}

We  construct Hecke grids on $\Gamma^*(1) = \Gamma_0(1)$ which begin with the forms $E_k(z)/\eta^r(z)$ for $k\in \{4,6,8,10,14\}$ and $r\in \{4,8,12,16,20\}$ (similar results hold for all positive integers $r\leq 24$, but to state them would require unwieldy notation).

Let $\nu$ be the multiplier system for $\eta^4(z)$ on $\Gamma^*(1)$. We compute using \eqref{eq:eta-mult} that if $A=\pmatrix{a}{b}{c}{d} \in \Gamma^*(1)$, then
\begin{equation} \label{eq:mult-eta-4}
	\nu(A) = \zeta_6^{(a+d)c-bd(c^2-1)-3c}.
\end{equation}
Here $\zeta_{m}:=e^{2\pi i/m}$.

\begin{theorem} \label{thm:level1}
Suppose that $k\in\{4,6,8,10,14\}$ and that $r \in \{4, 8, 12, 16, 20\}$.
Define $s/t = r/24$ in lowest terms and 
let $\ell \in \{0,1,2\}$ be the unique integer satisfying $12\ell+k-r \in \{0,4,6,8,10,14\}$.

\begin{enumerate}[\textup{(}a\textup{)}]
\item If $d>0$ and $d\equiv s\pmod{t}$ then there exist unique forms
\begin{equation} \label{eq:lev-1-f-d-def-1}
	f_{d} = q^{-d/t} + \sum_{\substack{n>0 \\ n\equiv -s\bmod{t}}} a_{d}(n) q^{n/t} \in M_{k-r/2}^!(\Gamma^*(1),\bar{\nu}^{r/4}).
\end{equation}
\item There exists a unique form
\begin{equation} \label{eq:f-tls-def}
	f_{t\ell-s} = q^{s/t-\ell} + \cdots \in S_{k-r/2}(\Gamma^*(1),\nu^{r/4}).
\end{equation}
Furthermore, if  $d>t\ell-s$ and $d\equiv -s\pmod{t}$ then there exist unique forms
	\begin{equation} \label{eq:lev-1-f-d-def-2}
		f_{d} = q^{-d/t} + \sum_{\substack{n>s-t\ell \\ n\equiv s\bmod{t}}} a_{d}(n) q^{n/t} \in M_{k-r/2}^{!}(\Gamma^*(1), \nu^{r/4}).
	\end{equation}
\item Suppose that $p$ is an odd prime. If $p^{n} \equiv 1 \pmod{t}$ then we have
	\begin{equation} \label{eq:hecke-1}
		f_{s} \big| T^{(t)}(p^{n}) = p^{(k-r/2-1)n} f_{p^{n}s}.
	\end{equation}
If $p^{n} \equiv -1 \pmod{t}$ then we have
	\begin{equation} \label{eq:hecke-2}
		f_{s} \big| T^{(t)}(p^{n}) =
		\begin{cases}
			p^{(k-r/2-1)n} f_{p^{n}s} + a_{s}(p^n) f_{-s} & \text{ if } \ell=0, \\
			p^{(k-r/2-1)n} f_{p^{n}s} & \text{ otherwise}.
		\end{cases}
	\end{equation}
\end{enumerate}
\end{theorem}

\begin{remark}
An analogue of Theorem \ref{thm:level1} with $1\leq r \leq 23$ is also true, with the following modifications. When $r\equiv 2 \pmod{4}$ the multiplier system of $\eta^r(z)$ includes the character $\pfrac{-1}{\bullet}$, and the case $k-r=12$ needs to be treated separately. When $r$ is odd, one uses the half-integral weight Hecke operators, and there are fewer cases  since $p^{2n} \equiv 1 \pmod{t}$ for all $n$.
\end{remark}

Before proving Theorem \ref{thm:level1}, we sketch the  proof of \eqref{conj2}.
\begin{proof}[Proof of (1.7)]
Note that $F_d(z) = f_d(6z)$ in the notation of Theorem \ref{thm:level1}.  
 We have the relation
\begin{equation} \label{Up-Tp-relation}
	F_1\big|U(p^n) = F_1\big|T(p^n) - \sum_{j=1}^{n} p^{j} F_1\big| U(p^{n-j}) \big| V(p^j).
\end{equation}
	The case $p\equiv 1 \pmod{3}$ follows in a  straightforward way by induction.

Suppose that $p\equiv 2 \pmod{3}$. If $n$ is even then \eqref{eq:hecke-1} gives $F_1\big|T(p^n) \equiv 0 \pmod{p^n}$. If $n$ is odd, induction shows that $a_1(p^n) \equiv 0 \pmod{p^{\frac {n-1}{2}}}$, so that  $F_1\big|T(p^n) \equiv 0 \pmod{p^{\floor{\frac n2}}}$ by \eqref{eq:hecke-2}. Using \eqref{Up-Tp-relation} we conclude that
\begin{align*}
	F_1 \big|U(p^n)  &\equiv 0 \pmod{p^\alpha}
\end{align*}
where $\alpha = \min\left\{\tfloor{\frac n2}, j + \tfloor{\frac {n-j}2}\right\} = \tfloor{\frac n2}$.
\end{proof}

\begin{proof}[Proof of Theorem \ref{thm:level1}]
Let $\Delta(z) := \eta^{24}(z)$ and let $j(z)$ denote the Hauptmodul on $\Gamma_{0}(1)$ given by
\[
	j(z) := \frac{E_{4}^{3}}{\Delta(z)} = q^{-1}+744+196884 q+21493760 q^2+864299970 q^3+\cdots \in M_{0}^{!}(\Gamma_{0}(1)).
\]

(a) Set $f_{s}(z) := E_{k}(z)/\eta^{r}(z)=q^{-s/t}+O(q^{1-s/t})$. For $d>s$ with  $d \equiv s\pmod{t}$ define
\[
	f_{d}(z) := j(z)^{(d-s)/t} f_{d-t}(z) + \sum_{m=2}^{(d-s)/t} c_{m} f_{d-mt}(z),
\]
where the $c_{m}$ are chosen so that $f_{d}(z) = q^{-d/t} + O(q^{1-s/t})$. These forms satisfy the requirements in \eqref{eq:lev-1-f-d-def-1}.
For uniqueness, suppose there are two forms $f_{d}$ and $f_{d}'$ satisfying \eqref{eq:lev-1-f-d-def-1} and define $g(z) := \eta^{r}(z)(f_{d}(z) - f_{d}'(z)) = O(q)$. Then $g(z)$ is in $S_{k}(\Gamma_0(1))$.  Since this space is trivial for $k\in \{4,6,8,10,14\}$, we conclude that $f_{d}=f_{d}'$.

(b) Set 
\[
	f_{t\ell-s} := \frac{E_{12\ell+k-r}(z)}{\Delta^\ell(z)} \eta^r(z) = q^{s/t-\ell} + O(q^{1+s/t-\ell}),
\]
where $E_0(z) := 1$, and set $f_{-s}=0$ if $\ell \neq 0$. For $d>t\ell-s$ with $d\equiv -s\pmod{t}$, define
\[
	f_{d}(z) := j(z)^{(d+s)/t} f_{d-t}(z) + \sum_{m=2}^{(d+s)/t-\ell} c_{m} f_{d-mt}(z),
\]
where the $c_{m}$ are chosen so that $f_{d}(z) = q^{-d/t} + O(q^{1+s/t-\ell})$. If there are two forms $f_{d}$ and $f_{d}'$ which each satisfy \eqref{eq:f-tls-def} or \eqref{eq:lev-1-f-d-def-2} then the form
\[
	g(z) := \Delta^{\ell}(z)\frac{f_{d}(z) - f_{d}'(z)}{\eta^{r}(z)} = O(q)
\]
has trivial multiplier system, so it is an element of $S_{12\ell+k-r}(\Gamma_0(1))$. This space is trivial since $12\ell+k-r\in \{0,4,6,8,10,14\}$, so $f_{d} = f_{d}'$.

(c) Since $rt/24=s\in\Z$ we see from \eqref{eq:mult-eta-4} that the multiplier system $\nu^{r/4}$  is trivial on $\Gamma_0(t,t)$ and takes values which are $t$-th roots of unity.  Therefore Proposition \ref{prop:Tp} gives
\[
	f_{s} \big| T_{k-r/2}^{(t)}(p^n) \in M_{k-r/2}^!(\Gamma^*(1),\bar{\nu}^{p^nr/4}).
\] 
It follows from this and \eqref{eq:T-q-series} that 
if $p^n \equiv 1 \pmod{t}$ then
\[
	f_{s} \big| T^{(t)}(p^{n}) = p^{(k-r/2-1)n} q^{-p^{n}s/t} + O(q^{1-s/t}) \in M_{k-r/2}^{!}(\Gamma^*(1),\bar{\nu}^{r/4}),
\]
while if $p^{n} \equiv -1 \pmod{t}$ then
\[
	f_{s} \big| T^{(t)}(p^{n}) - a_{s}(p^n) f_{-s} = p^{(k-r/2-1)n} q^{-p^{n}s/t} + O(q^{1+s/t-\ell}) \in M_{k-r/2}^{!}(\Gamma^*(1), \nu^{r/4}).
\]
By uniqueness we obtain \eqref{eq:hecke-1} and \eqref{eq:hecke-2}.
\end{proof}

\begin{example}
Computing as described in the proof above with $k=6$ and $r=4$, we obtain
\begin{align*}
	f_{1} &= q^{-\frac{1}{6}}-500 q^{\frac{5}{6}}-18634 q^{\frac{11}{6}}-196520 q^{\frac{17}{6}}-1277535
q^{\frac{23}{6}}-6146028 q^{\frac{29}{6}}+\cdots \\
	f_{7} &= q^{-\frac{7}{6}}-71750 q^{\frac{5}{6}}-86461760 q^{\frac{11}{6}}-13650854021
q^{\frac{17}{6}}-851755409792 q^{\frac{23}{6}}+\cdots \\
	f_{13} &= q^{-\frac{13}{6}}-2401000 q^{\frac{5}{6}}-24581234095 
q^{\frac{11}{6}}-19372032655696 q^{\frac{17}{6}} + \cdots \\
	f_{19} &= q^{-\frac{19}{6}}-44127125 q^{\frac{5}{6}}-2445793637760 
q^{\frac{11}{6}}-6837455343912760 q^{\frac{17}{6}} +\cdots
\end{align*}
and
\begin{align*}
	f_{5} &= q^{-\frac{5}{6}}-4 q^{\frac{1}{6}}-196882 q^{\frac{7}{6}}-42199976 
q^{\frac{13}{6}}-2421343603 q^{\frac{19}{6}}+\cdots \\
	f_{11} &= q^{-\frac{11}{6}}-14 q^{\frac{1}{6}}-22281280 q^{\frac{7}{6}}-40574734265 
q^{\frac{13}{6}}-12603830624640 q^{\frac{19}{6}}+\cdots \\
	f_{17} &= q^{-\frac{17}{6}}-40 q^{\frac{1}{6}}-953031331 q^{\frac{7}{6}}-8662803937424 
q^{\frac{13}{6}}-9545716711560680 q^{\frac{19}{6}}+\cdots \\
	f_{23} &= q^{-\frac{23}{6}}-105 q^{\frac{1}{6}}-24011843968 
q^{\frac{7}{6}}-837470540062104 q^{\frac{13}{6}}-2657886912184060160 
q^{\frac{19}{6}}+\cdots.
\end{align*}
\end{example}


\section{Hecke grids on $\Gamma^*(2)$} \label{sec:level-2}
In this section we construct grids on $\Gamma^*(2)$ which lead to the congruences \eqref{level2pm}.
Let
\begin{equation*}
	h_2(z) := \eta^2(z)\eta^2(2z) = q^\frac14-2 q^\frac54-3 q^\frac94+6 q^\frac{13}4 +\cdots.
\end{equation*}
The grids   begin with forms $f/h_2$, where $f \in M_4(\Gamma_0(2))$. This space is two-dimensional and is spanned by the forms
\begin{gather*}
	F_2^+(z) := \tfrac15\left(4E_4(2z)+E_4(z)\right) = 1+48 q+624 q^2+1344 q^3+\cdots,\\
	F_2^-(z) := \tfrac13\left(4E_4(2z)-E_4(z)\right) = 1-80 q-400 q^2-2240 q^3+\cdots.
\end{gather*}
Here $F_2^+$ and $F_2^-$ are eigenforms of $W_2$ with eigenvalues $\pm 1$, respectively. Since  $M_6(\Gamma_0(2))$ is also two-dimensional, the results in this section have analogues for $k=6$, using the eigenforms
\[
	G_2^\pm(z) := \frac{8E_6(2z) \pm E_6(z)}{8 \pm 1} \in M_6(\Gamma_0(2)).
\]
The details are similar, and are omitted.

Let $\nu_\pm$ denote the multiplier system for $h_2(z)$ on $\Gamma_0(2)$, extended to $\Gamma^*(2)$ via $\nu_\pm(W_2) = \pm 1$. If $\gamma=\pmatrix abcd \in \Gamma_0(2)$, then a computation involving \eqref{eq:eta-mult} gives
\begin{equation}
	\nu_\pm(\gamma) = i^{d(b-c/2)},
\end{equation}
which is trivial on $\Gamma_0(8,4)$.
We have 
\[
	h_2\in S_{2}(\Gamma^*(2), \nu_-).
\]

\begin{theorem}\label{level2thm}
\begin{enumerate}[\textup{(}a\textup{)}]
\item If $d>0$ and $d\equiv 1\pmod4$, then there exist unique forms
\begin{equation} \label{eq:2-d1mod4}
	f_d^{\pm}=q^{-d/4}+\sum_{\substack{n>0 \\ n\equiv 3\bmod4}} a_{d}^\pm(n)q^{n/4}\in  M_2^!(\Gamma^*(2), \overline{\nu}_\pm).
\end{equation}
\item If $d>0$ and $d\equiv 3\pmod4$, then there exist unique forms
\begin{equation} \label{eq:2-d3mod4p}
	f_d^{+}=q^{-d/4}+\sum_{\substack{n>0 \\ n\equiv 1\bmod4}} a_{d}^+(n)q^{n/4}\in M_2^!(\Gamma^*(2), \nu_+)
\end{equation}
and
\begin{equation} \label{eq:2-d3mod4m}
	f_d^{-}=q^{-d/4}+\sum_{\substack{n\geq5 \\ n\equiv 1\bmod4}} a_{d}^-(n)q^{n/4}\in M_2^{!}(\Gamma^*(2), \nu_-).
\end{equation}
\item For all odd prime powers $p^n$ we have 
\begin{equation*}
	f_1^+\big|T^{(4)}(p^n)=p^n f_{p^n}^+.
\end{equation*}
If $p^n\equiv 1\pmod 4$ then 
\begin{equation*}
	f_1^-\big|T^{(4)}(p^n)=p^n f_{p^n}^-.
\end{equation*}
If $p^n\equiv 3\pmod 4$ then
\begin{equation*}
	f_1^-\big|T^{(4)}(p^n)=p^nf_{p^n}^-+ a_1^-(p^n)\cdot h_2.
\end{equation*}
\end{enumerate}
\end{theorem}

\begin{proof}
 Let $j_2(z)$ denote the Hauptmodul on $\Gamma^*(2)$ given by
\[
	j_2(z) := \frac{\Delta(z)}{\Delta(2z)}+24+2^{12}\frac{\Delta(2z)}{\Delta(z)} = \frac{1}{q}+4372 q+96256 q^2+\cdots \in M_0^!(\Gamma^{*}(2)).
\]

Since $h_2$ has eigenvalue $-1$ under $W_2$, we define
\begin{gather*}
	f_1^+ := \frac{F_2^-}{h_2} = q^{-\frac{1}{4}}-78 q^{\frac{3}{4}}-553 q^{\frac{7}{4}}-3586 q^{\frac{11}{4}}-11325 q^{\frac{15}{4}}+\cdots \in M_2^!(\Gamma^*(2),\bar{\nu}_+), \\
	f_1^- := \frac{F_2^+}{h_2} = q^{-\frac{1}{4}}+50 q^{\frac{3}{4}}+727 q^{\frac{7}{4}}+2942 q^{\frac{11}{4}}+12995 q^{\frac{15}{4}}+\cdots \in M_2^!(\Gamma^*(2),\bar{\nu}_-).
\end{gather*}
For $d\equiv 1 \pmod 4$ we can construct $f_d^+$ satisfying  \eqref{eq:2-d1mod4} as a linear combination of $f_{d-4}\cdot j_2$ and $f_{d-4}, \dots, f_{1}$.   To prove uniqueness, suppose that $f_d^+$ and $g_d^+$ are two forms with these properties. 
Let $\omega_-$ be the multiplier system on $\Gamma^*(2)$ which maps $T$ to $1$ and $W_2$ to $-1$. Then 
\[
	h_2(z)(f_d^+(z)-g_d^+(z))=O(q)\in M_4^!(\Gamma^*(2), \omega_-).
\]
Since there is only one cusp, this is in fact a cusp form, and is therefore equal to zero.

The remaining forms are constructed in similar fashion.  When $d\equiv 3\pmod 4$, we begin with the forms
\begin{gather*}
	f_3^+ := \frac{F_2^+F_2^-}{h_2^3} = q^{-\frac34}-26 q^\frac14-3775 q^\frac54-92634 q^\frac94+\dots \in M_2^!(\Gamma^*(2),\nu_+),\\
	f_{-1}^- := h_2 = q^\frac14-2 q^\frac54-3 q^\frac94+6 q^\frac{13}4 +\dots \in M_2^!(\Gamma^*(2),\nu_-).
\end{gather*}
We conclude the proof by applying \eqref{eq:T-q-series} and Proposition \ref{prop:Tp} to the forms $f_1^\pm$ to obtain the equalities listed in (c).
\end{proof}

\begin{example}
We have
\begin{align*}
	f_{1}^+ &= q^{-\frac{1}{4}}-78 q^{\frac{3}{4}}-553 q^{\frac{7}{4}}-3586 q^{\frac{11}{4}}-11325 
q^{\frac{15}{4}}+\cdots\\
	f_{5}^+ &= q^{-\frac{5}{4}}-2265 q^{\frac{3}{4}}-291480 q^{\frac{7}{4}}-8976715 
q^{\frac{11}{4}}-155852328 q^{\frac{15}{4}}+\cdots\\
	f_{9}^+ &= q^{-\frac{9}{4}}-30878 q^{\frac{3}{4}}-16474122 q^{\frac{7}{4}}-1629968274 
q^{\frac{11}{4}}-71856917725 q^{\frac{15}{4}}+\cdots\\
	f_{13}^+ &= q^{-\frac{13}{4}}-232056 q^{\frac{3}{4}}-443763544 q^{\frac{7}{4}}-107298900269 
q^{\frac{11}{4}}-10015296762600 q^{\frac{15}{4}}+\cdots
\end{align*}
and
\begin{align*}
	f_{3}^+ &= q^{-\frac{3}{4}}-26q^{\frac{1}{4}}-3775 q^{\frac{5}{4}}-92634 q^{\frac{9}{4}}-1005576 
q^{\frac{13}{4}}-8083772 q^{\frac{17}{4}}+\cdots\\
	f_{7}^+ &= q^{-\frac{7}{4}}-79q^{\frac{1}{4}}-208200 q^{\frac{5}{4}}-21181014 
q^{\frac{9}{4}}-824132296 q^{\frac{13}{4}}+\cdots\\
	f_{11}^+ &= q^{-\frac{11}{4}}-326q^{\frac{1}{4}}-4080325 q^{\frac{5}{4}}-1333610406 
q^{\frac{9}{4}}-126807791227 q^{\frac{13}{4}}+\cdots\\
	f_{15}^+ &= q^{-\frac{15}{4}}-755q^{\frac{1}{4}}-51950776 q^{\frac{5}{4}}-43114150635 
q^{\frac{9}{4}}-8679923860920 q^{\frac{13}{4}}+\cdots,
\end{align*}
as well as
\begin{align*}
	f_{1}^- &= q^{-\frac{1}{4}}+50 q^{\frac{3}{4}}+727 q^{\frac{7}{4}}+2942 q^{\frac{11}{4}}+12995 
q^{\frac{15}{4}}+\cdots\\
	f_{5}^- &= q^{-\frac{5}{4}}+2599 q^{\frac{3}{4}}+281448 q^{\frac{7}{4}}+9097141 
q^{\frac{11}{4}}+154926040 q^{\frac{15}{4}}+\cdots\\
	f_{9}^- &= q^{-\frac{9}{4}}+29154 q^{\frac{3}{4}}+16632054 q^{\frac{7}{4}}+1625776110 
q^{\frac{11}{4}}+71919500835 q^{\frac{15}{4}}+\cdots\\
	f_{13}^- &= q^{-\frac{13}{4}}+238728 q^{\frac{3}{4}}+442272424 q^{\frac{7}{4}}+107373859795 
q^{\frac{11}{4}}+10013399068440 q^{\frac{15}{4}}+\cdots
\end{align*}
and
\begin{align*}
	f_{-1}^- &= q^{\frac{1}{4}}-2 q^{\frac{5}{4}}-3 q^{\frac{9}{4}}+6 q^{\frac{13}{4}}+2 
q^{\frac{17}{4}}+\cdots\\
	f_{3}^- &= q^{-\frac{3}{4}} + 4365 q^{\frac{5}{4}}+87512 q^{\frac{9}{4}}+1034388 q^{\frac{13}{4}}+7956216 
q^{\frac{17}{4}}+\cdots\\
	f_{7}^- &= q^{-\frac{7}{4}} + 201242 q^{\frac{5}{4}}+21384381 q^{\frac{9}{4}}+821362450 
q^{\frac{13}{4}}+18482815673 q^{\frac{17}{4}}+\cdots\\
	f_{11}^- &= q^{-\frac{11}{4}} + 4135599 q^{\frac{5}{4}}+1330181256 q^{\frac{9}{4}}+126896378153 
q^{\frac{13}{4}}+6154813925224 q^{\frac{17}{4}}+\cdots .
\end{align*}
\end{example}


\section{Hecke grids on $\Gamma^*(3)$} \label{sec:level-3}

Let
\begin{equation*}\label{def-h3}
	h_3(z):=\eta^2(z)\eta^2(3z) = q^{\frac13}-2 q^{\frac43}-q^{\frac73}+5 q^{\frac{13}3}+ \cdots.
\end{equation*}
We construct  grids on $\Gamma^*(3)$ starting with the forms $f/h_3$, where $f\in M_4(\Gamma_0(3))$. This space is two-dimensional, spanned by the $W_3$-eigenforms
\begin{gather*}
	F_3^+(z) := \tfrac{1}{10}(9E_4(3z) + E_4(z)) = 1+24 q+216 q^2+888 q^3+1752 q^4+\cdots, \\
	F_3^-(z) := \tfrac{1}{8}(9E_4(3z) - E_4(z)) = 1-30 q-270 q^2-570 q^3-2190 q^4+\cdots.
\end{gather*}

Let $\nu_\pm$ denote the multiplier system of $h_3(z)$ on $\Gamma_0(3)$, extended to $\Gamma^*(3)$ via $\nu_\pm(W_3) = \pm 1$. Using \eqref{eq:eta-mult}, we see that if $\gamma = \pmatrix abcd \in \Gamma_0(3)$, we have
\begin{equation}	\label{def-mult-3}
	\nu_\pm(\gamma) = \zeta_3^{\frac c3(a+d)+bd},
\end{equation}
which is trivial on $\Gamma_0(9,3)$.

\begin{theorem} \label{level3thm}
\begin{enumerate}[\textup{(}a\textup{)}]
\item If $d>0$ and $d\equiv 1\pmod3$, then there exist unique forms
\begin{equation*}
	f_d^{\pm}=q^{-d/3}+\sum_{\substack{n>0 \\ n\equiv 2\bmod3}} a_{d}^\pm(n)q^{n/3}\in  M_2^{!}(\Gamma^*(3), \bar{\nu}_\pm).
\end{equation*}
\item If $d>0$ and $d\equiv 2\pmod3$, then there exist unique forms
\begin{equation*}
	f_d^{+}=q^{-d/3}+\sum_{\substack{n>0 \\ \equiv 1\bmod3}} a_{d}^+(n)q^{n/3}\in M_2^{!}(\Gamma^*(3), \nu_+)
\end{equation*}
and
\begin{equation*}
	f_d^{-}=q^{-d/3}+\sum_{\substack{n\geq4 \\ n\equiv 1\bmod3}} a_{d}^-(n)q^{n/3}\in M_2^{!}(\Gamma^*(3), \nu_-).
\end{equation*}
\item Suppose $p\geq 5$ is prime. We have 
\begin{equation*}
	f_1^+\big|T^{(3)}(p^n)=p^n f_{p^n}^+.
\end{equation*}
If $p^n\equiv 1\pmod 3$ then 
\begin{equation*}
	f_1^-\big|T^{(3)}(p^n)=p^n f_{p^n}^-.
\end{equation*}
If $p^n\equiv 2\pmod 3$ then
\begin{equation*}
	f_1^-\big|T^{(3)}(p^n)=p^nf_{p^n}^-+a_1^-(p^n)\cdot h_3.
\end{equation*}
\end{enumerate}
\end{theorem}

\begin{proof}
Let $\omega_-$ be the multiplier which maps $W_3$ to $-1$, and 
define
\[
	G_3^-(z) := \tfrac{1}{2} \left( E_2(3z) - E_2(z) \right) = 1+12 q+36 q^2+12 q^3+84 q^4+ \cdots \in M_2(\Gamma^*(3), \omega_-).
\]
The four grids are constructed beginning with the forms
\begin{gather*}
	f_1^+ := \frac{F_3^-}{h_3} = q^{-\frac{1}{3}} - 28q^{\frac{2}{3}} - 325q^{\frac{5}{3}} - 1248q^{\frac{8}{3}} - 5016q^{\frac{11}{3}} + \cdots \in M_2^!(\Gamma^*(3), \bar{\nu}_+), \\
	f_1^- := \frac{F_3^+}{h_3} = q^{-\frac{1}{3}} + 26q^{\frac{2}{3}} + 269q^{\frac{5}{3}} + 1452q^{\frac{8}{3}} + 4920q^{\frac{11}{3}} + \cdots \in M_2^!(\Gamma^*(3), \bar{\nu}_-), \\
	f_2^+ := \frac{F_3^- G_3^-}{h_3^2} =  q^{-\frac{2}{3}} - 14q^{\frac{1}{3}} - 652q^{\frac{4}{3}} - 7462q^{\frac{7}{3}} - 47525q^{\frac{10}{3}} + \cdots \in M_2^!(\Gamma^*(3), \nu_+), \\
	f_{-1}^- := h_3 = q^{\frac{1}{3}}-2 q^{\frac{4}{3}}-q^{\frac{7}{3}}+5 q^{\frac{13}{3}} + 4q^{\frac{16}{3}} + \cdots \in M_2^!(\Gamma^*(3), \nu_-).
\end{gather*}
The remaining forms $f_d^\pm$ are constructed using the Hauptmodul $j_3(z)$ on $\Gamma^*(3)$ given by
\begin{equation*}
	j_3(z) = \frac{\eta^{12}(z)}{\eta^{12}(3z)} + 12 + 3^6 \frac{\eta^{12}(3z)}{\eta^{12}(z)} = q^{-1} + 783q + 8672q^2 + \cdots \in M_0^!(\Gamma^*(3)).
	\qedhere
\end{equation*}
\end{proof}

\begin{example}
We have
\begin{align*}
	f_{1}^+ &= q^{-\frac{1}{3}}-28 q^{\frac{2}{3}}-325 q^{\frac{5}{3}}-1248 q^{\frac{8}{3}}-5016
q^{\frac{11}{3}}+\cdots\\
	f_{4}^+ &= q^{-\frac{4}{3}}-326 q^{\frac{2}{3}}-23600 q^{\frac{5}{3}}-471884 q^{\frac{8}{3}}-5409712
q^{\frac{11}{3}}+\cdots\\
	f_{7}^+ &= q^{-\frac{7}{3}}-2132 q^{\frac{2}{3}}-513250 q^{\frac{5}{3}}-25773728
q^{\frac{8}{3}}-636531533 q^{\frac{11}{3}}+\cdots\\
	f_{10}^+ &= q^{-\frac{10}{3}}-9505 q^{\frac{2}{3}}-6467264 q^{\frac{5}{3}}-677506240
q^{\frac{8}{3}}-30773378240 q^{\frac{11}{3}}+\cdots
\end{align*}
and
\begin{align*}
	f_{2}^+ &= q^{-\frac{2}{3}}-14 q^{\frac{1}{3}}-652 q^{\frac{4}{3}}-7462 q^{\frac{7}{3}}-47525
q^{\frac{10}{3}}+\cdots\\
	f_{5}^+ &= q^{-\frac{5}{3}}-65 q^{\frac{1}{3}}-18880 q^{\frac{4}{3}}-718550
q^{\frac{7}{3}}-12934528 q^{\frac{10}{3}}+\cdots\\
	f_{8}^+ &= q^{-\frac{8}{3}}-156 q^{\frac{1}{3}}-235942 q^{\frac{4}{3}}-22552012
q^{\frac{7}{3}}-846882800 q^{\frac{10}{3}}+\cdots\\
	f_{11}^+ &= q^{-\frac{11}{3}}-456 q^{\frac{1}{3}}-1967168 q^{\frac{4}{3}}-405065521
q^{\frac{7}{3}}-27975798400 q^{\frac{10}{3}}+\cdots,
\end{align*}
as well as
\begin{align*}
	f_{1}^- &= q^{-\frac{1}{3}}+26 q^{\frac{2}{3}}+269 q^{\frac{5}{3}}+1452 q^{\frac{8}{3}}+4920 q^{\frac{11}{3}}+\cdots\\
	f_{4}^- &= q^{-\frac{4}{3}}+376 q^{\frac{2}{3}}+23488 q^{\frac{5}{3}}+468634 q^{\frac{8}{3}}+5427008
q^{\frac{11}{3}}+\cdots\\
	f_{7}^- &= q^{-\frac{7}{3}}+2026 q^{\frac{2}{3}}+516638 q^{\frac{5}{3}}+25767436
q^{\frac{8}{3}}+636345829 q^{\frac{11}{3}}+\cdots\\
	f_{10}^- &= q^{-\frac{10}{3}}+9449 q^{\frac{2}{3}}+6456448 q^{\frac{5}{3}}+677710592
q^{\frac{8}{3}}+30773024128 q^{\frac{11}{3}}+\cdots
\end{align*}
and
\begin{align*}
	f_{-1}^- &= q^{\frac{1}{3}}-2 q^{\frac{4}{3}}-q^{\frac{7}{3}}+5 q^{\frac{13}{3}} + 4q^{\frac{16}{3}}+\cdots\\
	f_{2}^- &= q^{-\frac{2}{3}}+778 q^{\frac{4}{3}}+7104 q^{\frac{7}{3}}+47245 q^{\frac{10}{3}}+232128
q^{\frac{13}{3}}+\cdots\\
	f_{5}^- &= q^{-\frac{5}{3}}+18898 q^{\frac{4}{3}}+723347 q^{\frac{7}{3}}+12912896
q^{\frac{10}{3}}+152125263 q^{\frac{13}{3}}+\cdots\\
	f_{8}^- &= q^{-\frac{8}{3}}+234680 q^{\frac{4}{3}}+22546688 q^{\frac{7}{3}}+847138240
q^{\frac{10}{3}}+18799619328 q^{\frac{13}{3}}+\cdots.
\end{align*}
\end{example}


\section{Hecke grids on $\Gamma^*(4)$} \label{sec:level-4}
The three-dimensional space $M_4(\Gamma_0(4))$ is spanned by $\{E_4(2z), F_4^+(z), F_4^-(z)\}$, where
\begin{gather*}
	F_4^+(z) := \tfrac{1}{15}(16E_4(4z) + E_4(z) - 2E_4(2z)), \\
	F_4^-(z) := \tfrac{1}{15}(16E_4(4z) - E_4(z)).
\end{gather*}
The forms $E_4(2z)$ and $F_4^+(z)$ have eigenvalue $+1$ under the Fricke involution $W_4$, while the form $F_4^-$ has eigenvalue $-1$. Let
\[
	h_4(z) := \eta^4(2z) = q^{\frac{1}{3}} -4 q^{\frac{7}{3}}+2 q^{\frac{13}{3}}+8 q^{\frac{19}{3}}-5 q^{\frac{25}{3}} + \cdots.
\]
We construct grids on $\Gamma^*(4)$ starting with forms $f(z)/h_4$, where $f(z) \in M_4(\Gamma_0(4))$.

Recall that $E_4(z)/\eta^4(z)$ is the first member of one of the $\Gamma^*(1)$ grids. So we need concern ourselves only with the subspace spanned by $\{F_4^+, F_4^-\}$.
The distinguishing feature of  $F_4^+$ is the fact that it vanishes to order 2 at the cusp $1/2$.

Let $\nu_\pm$ denote the multiplier system for $\eta^4(2z)$ on $\Gamma_0(4)$, extended to $\Gamma^*(4)$ by $\nu_\pm(W_4)=\pm 1$. If $\gamma = \pmatrix abcd \in \Gamma_0(4)$, then by applying \eqref{eq:mult-eta-4} to the matrix $\pmatrix{a}{2b}{c/2}{d} = A_2 \gamma A_2^{-1}$ we obtain
\begin{equation}	\label{def-mult-4}
	\nu_\pm(\gamma) = \zeta_3^{bd(1-(c/2)^2) + \frac c4 (a+d)}.
\end{equation}
Note that $\nu_\pm$ is trivial on $\Gamma_0(12,3)$.
Since  $\eta^4(2z)\big|_2W_4=-\eta^4(2z)$, we have $h_4 \in S_2(\Gamma^*(4), \nu_-)$.

\begin{theorem} \label{level4thm}
\begin{enumerate}[\textup{(}a\textup{)}]
\item If $d>0$ and $d\equiv 1\pmod3$, then there exist unique forms
\begin{equation}
	f_d^{+}=q^{-d/3}+\sum_{\substack{n>0 \\ n\equiv 2\bmod3}} a_{d}^+(n)q^{n/3}\in  M_2^{!}(\Gamma^*(4), \bar{\nu}_+).
\end{equation}
Furthermore, there exist unique forms
\begin{equation} \label{eq:level-4-f-d-vanish}
	f_d^{-}=q^{-d/3}+\sum_{\substack{n>0 \\ n\equiv 2\bmod3}} a_{d}^-(n)q^{n/3}\in  M_2^{!}(\Gamma^*(4), \bar{\nu}_-)
\end{equation}
which vanish at the cusp $1/2$.
\item If $0<d\equiv 2\pmod3$, then there exist unique forms
\begin{equation}
	f_d^{+}=q^{-d/3}+\sum_{\substack{n>0 \\ n\equiv 1\bmod3}} a_{d}^+(n)q^{n/3}\in M_2^{!}(\Gamma^*(4), \nu_+)
\end{equation}
and
\begin{equation}
	f_d^{-}=q^{-d/3}+\sum_{\substack{n\geq4 \\ n\equiv 1\bmod3}}a_{d}^-(n)q^{n/3}\in M_2^{!}(\Gamma^*(4), \nu_-).
\end{equation}
\item Suppose $p\geq 5$ is prime. We have
\begin{equation*}
	f_1^+\big|T^{(3)}(p^n)=p^n f_{p^n}^+.
\end{equation*}
If $p^n\equiv 1\pmod 3$ then 
\begin{equation*}
	f_1^-\big|T^{(3)}(p^n)=p^n f_{p^n}^-.
\end{equation*}
If $p^n\equiv 2\pmod 3$ then
\begin{equation*}
	f_1^-\big|T^{(3)}(p^n)=p^nf_{p^n}^-+a_1^-(p^n)\cdot h_4.
\end{equation*}
\end{enumerate}
\end{theorem}

\begin{proof}
Let
\[
	G_4^-(z) := \tfrac{1}{3}\left( 4E_2(4z) - E_2(z) \right) = 1+8 q+24 q^2+32 q^3+24 q^4+ \cdots \in M_2(\Gamma^*(2), \omega_-)
\]
The four grids are constructed beginning with the forms
\begin{gather*}
	f_1^+ := \frac{F_4^-}{h_4} = q^{-\frac{1}{3}}-16 q^{\frac{2}{3}}-140 q^{\frac{5}{3}}-512 q^{\frac{8}{3}} - 1474q^{\frac{11}{3}} + \cdots \in M_2^!(\Gamma^*(4), \bar{\nu}_+), \\
	f_1^- := \frac{F_4^+}{h_4} = q^{-\frac{1}{3}}+16 q^{\frac{2}{3}}+116 q^{\frac{5}{3}}+512 q^{\frac{8}{3}}+1598 q^{\frac{11}{3}} + \cdots \in M_2^!(\Gamma^*(4), \bar{\nu}_-), \\
	f_2^+ := \frac{F_4^- G_4^-}{h_4^2} = q^{-\frac{2}{3}}+ 8q^{\frac{1}{3}} - 240 q^{\frac{4}{3}} - 2016q^{\frac{7}{3}} - 10380q^{\frac{10}{3}} + \cdots \in M_2^!(\Gamma^*(4), \nu_+), \\
	f_{-1}^- := h_4 = q^{\frac{1}{3}} -4 q^{\frac{7}{3}}+2 q^{\frac{13}{3}}+8 q^{\frac{19}{3}}-5 q^{\frac{25}{3}} + \cdots \in M_2^!(\Gamma^*(4), \nu_-).
\end{gather*}
The remaining forms $f_d^\pm$ are constructed using the Hauptmodul $j_4(z)$ on $\Gamma^*(4)$ given by
\begin{equation*}
	j_4(z) := \frac{\eta^8(z)}{\eta^8(4z)} + 8 + \frac{\eta^8(4z)}{\eta^8(z)} = \frac{1}{q}+276 q+2048 q^2 + \cdots \in M_0^!(\Gamma^*(4)).
\end{equation*}
For $d\equiv 1\pmod{3}$, the forms $f_d^-$ are constructed so that they vanish at $1/2$. This property is necessary to establish uniqueness, for if $f_d^-$ and $g_d^-$ satisfy  \eqref{eq:level-4-f-d-vanish} then
\[
	h_4 \cdot (f_d^- - g_d^-)=O(q)
\]
vanishes at $\infty$ and vanishes to order $2$ at $1/2$. But nonzero weight $4$ forms on $\Gamma_0(4)$ can have at most $2$ zeros, so $f_d^-=g_d^-$.
\end{proof}

\begin{example}
We have
\begin{align*}
	f_{1}^+ &= q^{-\frac{1}{3}}-16 q^{\frac{2}{3}}-140 q^{\frac{5}{3}}-512 q^{\frac{8}{3}}-1474
q^{\frac{11}{3}}+\cdots\\
	f_{4}^+ &= q^{-\frac{4}{3}}-120 q^{\frac{2}{3}}-5120 q^{\frac{5}{3}}-69872 q^{\frac{8}{3}}-585728
q^{\frac{11}{3}}+\cdots\\
	f_{7}^+ &= q^{-\frac{7}{3}}-576 q^{\frac{2}{3}}-69950 q^{\frac{5}{3}}-2115584 q^{\frac{8}{3}}-34400960
q^{\frac{11}{3}}+\cdots\\
	f_{10}^+ &= q^{-\frac{10}{3}}-2076 q^{\frac{2}{3}}-606208 q^{\frac{5}{3}}-34664448
q^{\frac{8}{3}}-955187200 q^{\frac{11}{3}}+\cdots
\end{align*}
and
\begin{align*}
	f_{2}^+ &= q^{-\frac{2}{3}}-8 q^{\frac{1}{3}}-240 q^{\frac{4}{3}}-2016 q^{\frac{7}{3}}-10380
q^{\frac{10}{3}}+\cdots\\
	f_{5}^+ &= q^{-\frac{5}{3}}-28 q^{\frac{1}{3}}-4096 q^{\frac{4}{3}}-97930 q^{\frac{7}{3}}-1212416
q^{\frac{10}{3}}+\cdots\\
	f_{8}^+ &= q^{-\frac{8}{3}}-64 q^{\frac{1}{3}}-34936 q^{\frac{4}{3}}-1851136
q^{\frac{7}{3}}-43330560 q^{\frac{10}{3}}+\cdots\\
	f_{11}^+ &= q^{-\frac{11}{3}}-134 q^{\frac{1}{3}}-212992 q^{\frac{4}{3}}-21891520
q^{\frac{7}{3}}-868352000 q^{\frac{10}{3}}+\cdots,
\end{align*}
as well as
\begin{align*}
	f_{1}^- &= q^{-\frac{1}{3}} + 16 q^{\frac{2}{3}}+116 q^{\frac{5}{3}}+512 q^{\frac{8}{3}}+1598
q^{\frac{11}{3}}+\cdots\\
	f_{4}^- &= q^{-\frac{4}{3}} + 136 q^{\frac{2}{3}}+5120 q^{\frac{5}{3}}+69392 q^{\frac{8}{3}}+585728
q^{\frac{11}{3}}+\cdots\\
	f_{7}^- &= q^{-\frac{7}{3}} + 576 q^{\frac{2}{3}}+70338 q^{\frac{5}{3}}+2115584 q^{\frac{8}{3}}+34391360
q^{\frac{11}{3}}+\cdots\\
	f_{10}^- &= q^{-\frac{10}{3}} + 2020 q^{\frac{2}{3}}+606208 q^{\frac{5}{3}}+34672640
q^{\frac{8}{3}}+955187200 q^{\frac{11}{3}}+\cdots
\end{align*}
and
\begin{align*}
	f_{-1}^- &= q^{\frac{1}{3}}-4 q^{\frac{7}{3}}+2 q^{\frac{13}{3}}+8 q^{\frac{19}{3}}-5 q^{\frac{25}{3}}-4 q^{\frac{31}{3}}+\cdots\\
	f_{2}^- &= q^{-\frac{2}{3}} + 272 q^{\frac{4}{3}}+2048 q^{\frac{7}{3}}+10100 q^{\frac{10}{3}}+40960
q^{\frac{13}{3}}+\cdots\\
	f_{5}^- &= q^{-\frac{5}{3}} + 4096 q^{\frac{4}{3}}+98566 q^{\frac{7}{3}}+1212416 q^{\frac{10}{3}}+10351552
q^{\frac{13}{3}}+\cdots\\
	f_{8}^- &= q^{-\frac{8}{3}} + 34696 q^{\frac{4}{3}}+1851392 q^{\frac{7}{3}}+43340800
q^{\frac{10}{3}}+641007616 q^{\frac{13}{3}}+\cdots.
\end{align*}
\end{example}


\bibliographystyle{plain}
\bibliography{bibliography}

\begin{thebibliography}{1}

\bibitem{Ahlgren:2012}
Scott Ahlgren.
\newblock Hecke relations for traces of singular moduli.
\newblock {\em Bull. Lond. Math. Soc.}, 44(1):99--105, 2012.

\bibitem{Garthwaite}
Sharon~Anne Garthwaite.
\newblock Convolution congruences for the partition function.
\newblock {\em Proc. Amer. Math. Soc.}, 135(1):13--20, 2007.

\bibitem{Guerzhoy}
P.~Guerzhoy.
\newblock On the {H}onda-{K}aneko congruences.
\newblock In {\em From {F}ourier analysis and number theory to radon transforms
  and geometry}, volume~28 of {\em Dev. Math.}, pages 293--302. Springer, New
  York, 2013.

\bibitem{HK}
Yutaro Honda and Masanobu Kaneko.
\newblock On {F}ourier coefficients of some meromorphic modular forms.
\newblock {\em Bull. Korean Math. Soc.}, 49(6):1349--1357, 2012.

\bibitem{Knopp}
Marvin~I. Knopp.
\newblock {\em Modular functions in analytic number theory}.
\newblock Markham Publishing Co., Chicago, Ill., 1970.

\bibitem{Kohler}
G{{\"u}}nter K{{\"o}}hler.
\newblock {\em Eta products and theta series identities}.
\newblock Springer Monographs in Mathematics. Springer, Heidelberg, 2011.

\bibitem{Zagier:2002}
Don Zagier.
\newblock Traces of singular moduli.
\newblock In {\em Motives, polylogarithms and {H}odge theory, {P}art {I}
  ({I}rvine, {CA}, 1998)}, volume~3 of {\em Int. Press Lect. Ser.}, pages
  211--244. Int. Press, Somerville, MA, 2002.

\end{thebibliography}
\end{document}